\documentclass[11pt, oneside]{article}   	
\usepackage{geometry}                		
\usepackage{graphicx}				
\usepackage{amssymb}
\usepackage{amsmath}
\usepackage{amsthm}
\usepackage{marvosym}
\usepackage{esint}
\title{Stabilization rates for the damped wave equation with H\"older-regular damping}
\author{Perry Kleinhenz}
\date{}							

\theoremstyle{definition}
\newtheorem{definition}{Definition}

\theoremstyle{theorem}
\newtheorem{theorem}{Theorem}[section]
\newtheorem{lemma}[theorem]{Lemma}

\newtheorem{proposition}[theorem]{Proposition}

\newcommand{\Rb}{\mathbb{R}}
\newcommand{\Zb}{\mathbb{Z}}

\newcommand{\C}{\mathbb{C}}

\newcommand{\Nb}{\mathbb{N}}

\newcommand{\ra}{\rightarrow}

\newcommand{\e}{\varepsilon}
\renewcommand{\d}{\delta}

\newcommand{\nm}[1]{\left| \left| #1 \right| \right|}

\newcommand{\lp}[2]{ \nm{#1}_{L^{#2}}}

\newcommand{\hp}[2]{\nm{#1}_{H^{#2}}}

\newcommand{\Del}{\Delta}

\newcommand{\Cs}{C^{\infty}_0}

\newcommand{\p}{\partial}

\newcommand{\Ci}{C^{\infty}}

\newcommand{\T}{\mathbb{T}}

\newcommand{\ti}{\widetilde}

\newcommand{\ltb}[1]{\nm{#1}_{L^2(0,b)}}

\newcommand{\Rl}{\text{Ref}(\lambda_h,h)}
\newcommand{\Rm}{\text{Ref}(\mu_h,h)}
\begin{document}
\maketitle
\begin{abstract}
We study the decay rate of the energy of solutions to the damped wave equation in a setup where the geometric control condition is violated. We consider damping coefficients which are $0$ on a strip and vanish like polynomials, $x^{\beta}$. We prove that the semigroup cannot be stable at rate faster than $1/t^{(\beta+2)/(\beta+3)}$ by producing quasimodes of the associated stationary damped wave equation. We also prove that the semigroup is stable at rate at least as fast as $1/t^{(\beta+2)/(\beta+4)}$. These two results establish an explicit relation between the rate of vanishing of the damping and rate of decay of solutions. 
Our result partially generalizes a decay result of Nonnemacher in which the damping is an indicator function on a strip. 
\end{abstract}

\section{Introduction}

Let $M=(M,g)$ be a Riemannian manifold. Fix some $W \in L^{\infty}(M), W \geq 0$. We study the asymptotic behavior as $t \ra \infty$ of solutions to the damped wave equation 
\begin{equation} \label{DWE}
\begin{cases}
\p_t^2  u - \Del u + W(x) \p_t u = 0 & \text{in } M \times \Rb^+  \\
(u,\p_t u) |_{t=0} = (u_0, u_1) & \text{in } M \\
u|_{\p M}=0  &\text{if } \p M \neq \emptyset.
\end{cases}
\end{equation}
The quantity of particular interest is the energy 
\begin{equation}
E(u,t) = \frac{1}{2} \left(\lp{\nabla u(\cdot, t)}{2}^2 + \lp{\p_t u(\cdot, t) }{2}^2 \right).
\end{equation}
In this paper we will work on the square $M =[-b,b] \times [-b,b]$ or torus $M=\T^2$, which we parametrize by $(x,y)$. We will give detailed proofs in the case of the square then show how those results can be applied to the torus. 

For some fixed $\beta \geq 0$ and $a,\sigma>0$, such that $a+\sigma<b$, we study damping $W \in L^{\infty}(M)$ of the form
\begin{equation}\label{dampingdef}
W(x,y) = \begin{cases}
0 & |x|<a\\
(|x|-a)^{\beta} & a<|x|<a+\sigma \\
c(|x|)>0 & a+\sigma< |x|<b, 
\end{cases}
\end{equation}
In particular note that the damping is invariant in the $y$ direction. 

\textbf{Remark.} The particular form of $c(x)$ does not affect our result so long as $c(x) \in L^{\infty}(a+\sigma,b)$ and it is uniformly bounded away from 0. However choosing a $c$ such that $W$ is smooth for $|x|>a$ and $W=C>0$ for $|x|>a+2\sigma$ is perhaps the most interesting case in the context of existing results. 

\begin{definition}
Let $f(t)$ be a function such that $f(t) \ra 0$ as $t \ra \infty$. We say that \eqref{DWE} is stable at rate $f(t)$ if there exists a constant $C>0$ such that for all $(u_0, u_1) \in (H^2(M) \cap H_0^1(M)) \times H_0^1(M)$ (or
$H^2(M) \times H^1(M)$ if $\p M =\emptyset$) if $u$ solves \eqref{DWE} with $(u_0, u_1)$ as Cauchy data then
$$
E(u,t) \leq C f(t)^2 \left( \nm{u_0}_{H^2(M)}^2 + \nm{u_1}_{H^1(M)}^2 \right)\quad \text{ for all } t >0. 
$$
\end{definition}
Our main result is 
\begin{theorem}\label{mainresult}
For  all $\e>0$, with $W$ as in \eqref{dampingdef} the equation \eqref{DWE} cannot be stable at rate 
$$
t^{-\frac{\beta+2}{\beta+3}-\e}.
$$
\end{theorem}
More precisely if we use $m_1(t)$ to denote the best possible $f$ for which definition 1 holds then this result along with Lemma 4.6 of \cite{AnantharamanLeautaud2014} and Proposition 3 of \cite{BattyDuyckaerts2008} show that $m_1(t) \geq C/(1+t)^{\frac{\beta+2}{\beta+3}}$ for some $C>0$, where we use the notation of \cite{BattyDuyckaerts2008}.

We also show that 
\begin{theorem}\label{lowerbound}
For $W$ as in \eqref{dampingdef} with $\beta>0$ the equation \eqref{DWE} is stable at rate 
$$
t^{-\frac{\beta+2}{\beta+4}}.
$$
\end{theorem}
Again using Lemma 4.6 of \cite{AnantharamanLeautaud2014} and Proposition 3 of \cite{BattyDuyckaerts2008} this shows  $m_1(t) \leq C/(1+t)^{\frac{\beta+2}{\beta+4}}$ for some $C>0$. 

The decay of energy of the damped wave equation is a well studied question. The strongest possible decay is uniform stabilization, which is defined as the existence of $F(t) \ra 0$ as $t \ra \infty$ such that, for all $u$ solving \eqref{DWE} with initial data in $H^1(M) \times L^2(M)$ 
$$
E(u,t) \leq F(t) E(u,0), \quad t \geq 0.
$$  

It was established by \cite{RauchTaylor1975} that uniform stabilization occurs with $F=Ce^{- \kappa t}$ for some $\kappa, C>0$, when $\p M = \emptyset$, $W \in C^0(M)$ and $\text{supp }W$ satisfies the geometric control condition (GCC).  We recall that a set $U$ satisfies the GCC if there exists $T>0$ such that for every geodesic $\gamma$ on $M$ of length $T$,  $\gamma \cap U \neq \emptyset$. The reverse implication, that uniform stabilization with any $F$ implies $\text{supp } W$ satisfies the GCC, was shown in \cite{Ralston1969}. These results were extended to the case $M \neq \emptyset$ by \cite{BardosLebeauRauch1992} and \cite{BurqGerard1997}. This in turn guarantees that when uniform stabilization occurs one can always let $F=Ce^{-\kappa t}$ for some $\kappa,C>0$. For a finer discussion of when uniform stabilization occurs for $L^{\infty}$ damping see \cite{BurqGerard2018}.

A natural next question to ask is what occurs when the GCC does not hold for $\text{supp } W$. Because of the necessity of the GCC for uniform stabilization, as soon as it does not hold we must adjust the kind of decay we hope for. The next best thing is the stability defined in Definition 1, which comes from \cite{Lebeau1996} and requires initial data with an additional spatial derivative. 

In \cite{Burq1998}, the author showed that the energy of a solution to \eqref{DWE} decays at least logarithmically ($f(t)=1/\log(2+t)$) as soon as the damping $W(x) \geq c >0$ on some open, nonempty set. 
In \cite{Lebeau1996} the author gave explicit examples of domains on which $f(t)=1/\log(2+t)$ is the exact decay rate, in particular when $M=S^2$ and $\{W>0\}$ does not intersect a neighborhood of the equator. For related work see also \cite{LebeauRobbiano1997}.

In the case of the square when the damped region contains a vertical strip, \cite{LiuRao2005} established a decay rate of $f(t)=(\log(t)/t)^{1/2}$. This was expanded to the case of partially rectangular domains when $\{W>0\}$ contains a neighborhood of the nonrectangular part in \cite{BurqHitrik2007}. Additionally in \cite{BurqHitrik2007} a relation between vanishing rate of the damping and decay rate for the damped wave equation was established.  

These results were refined by \cite{AnantharamanLeautaud2014}.
The authors established a decay rate of $f(t)=1/t^{1/2}$ for the damped wave equation in a more general setting, so long as the associated Schr\"odinger equation is controllable. This includes the case of not identically vanishing damping on the 2 dimensional square (or torus) as a consequence of \cite{Jaffard1990} (resp. \cite{Macia2010}, \cite{BurqZworski2012}).

Continuing in the case of the 2 dimensional square \cite{AnantharamanLeautaud2014} also show that the system can not be stable at rate $f(t)=1/t^{1+\e}$ for any $\e>0$, when $\{W>0\}$ does not satisfy GCC,  (this condition is  referred to as the GCC being strongly violated).  They further show the existence of a smooth damping coefficient which strongly violates the GCC for which the energy decays at rate $f(t)=1/t^{1-\e}$ for any $\e>0$. 

In an appendix to \cite{AnantharamanLeautaud2014}, Nonnenmacher shows that when the damping is the indicator function on a strip on the square or torus the system cannot be stable at rate $f(t)=1/t^{2/3+\e}$, for any $\e>0$. The complementary result was shown in \cite{Stahn2017}  to establish $f(t)=1/t^{2/3}$ as the exact rate of decay when damping is a strip on the square or torus. The difference in behavior between the smooth and discontinuous damping led the authors of \cite{AnantharamanLeautaud2014} to pose the problem of establishing an explicit relation between the vanishing rate of the damping and the decay rate.  

An explicit relation was established by \cite{LeataudLerner2017} in a slightly different setting. The authors study the case of a general manifold in which the damping is supported everywhere but a flat subtorus. In the 2 dimensional case this is an example of the GCC not holding but also not being strongly violated. The damping is required to be invariant along this subtorus and must satisfy $W(x) \leq C|x|^{\beta}$ near where it vanishes. When this is the case the authors show that the system cannot be stable at rate  $f(t)=1/t^{\frac{\beta+2}{\beta}+\e}$, for any $\e>0$. 
They also show that if the vanishing behavior of the damping is further limited to $C_1^{-1}|x|^{\beta} \leq W(x) \leq C_1 |x|^{\beta}$ the system is stable at exactly the rate $f(t)=1/t^{\frac{\beta+2}{\beta}}$ (See also \cite{BurqZuily2015}). 

Note that in \cite{LeataudLerner2017} decreasing $\beta$ corresponds to faster vanishing (i.e. less regular damping) which produces faster decay, which is counter to the behavior exhibited in \cite{AnantharamanLeautaud2014}, \cite{BurqHitrik2007} and our own result, namely that faster vanishing (i.e. less regular damping) produces slower decay. However the dynamics in 2 dimensions in the two situations are different, with only one undamped orbit in the former as opposed to a whole family in the latter.

Our paper provides a partial answer to the problem posed by \cite{AnantharamanLeautaud2014}. We establish an explicit relation between the rate of vanishing of the damping and the stability rate of the system, in a case where the GCC is strongly violated on the square or $\T^2$. Our work also partially extends that of Nonnenmacher in the appendix to \cite{AnantharamanLeautaud2014}, which agrees with our first theorem when $\beta=0$, although that  result follows from the existence of modes of the stationary equation where we produce quasimodes. 
Our two results provide further evidence for the fact, discussed in \cite{AnantharamanLeautaud2014}, \cite{LeataudLerner2017}, \cite{BurqHitrik2007}, that once the support of the damping is fixed the rate of vanishing of the damping is the most significant feature when determining the decay rate. 

We note that our second result improves that of Theorem 1.2 of \cite{BurqHitrik2007}, which gives a decay rate of $f(t)=1/t^{\beta/(\beta+4)}$ (see also \cite{AnantharamanLeautaud2014} Theorem 2.6). We also note that there is a gap between our two results. Closing this gap would be an interesting area for further work. 

In the next section we outline the proof of Theorem \ref{mainresult}. Sections \ref{SectionCAPgoodenough}, \ref{SectionCAPeigen} and \ref{proofofdampedltwo} contain the details of the proof. Section \ref{prooflowerbound} contains the proof of Theorem \ref{lowerbound}. 

\textbf{Acknowledgements} The author would like to thank Jared Wunsch for his advice and guidance. The author would also like to thank Matthieu L\'eautaud for comments on an early manuscript. The author would also like to thank the referees for their prompt, detailed and constructive comments.  
This research was supported in part by the National Science Foundation grant ''RTG: Analysis on manifolds" at Northwestern University.

\section{Outline of Proof of Theorem \ref{mainresult}}

To prove Theorem \ref{mainresult} we rely on the following result from \cite{AnantharamanLeautaud2014} (Proposition 2.4) and \cite{BorichevTomilov2010} which relates energy decay of the damped wave equation to resolvent estimates of the stationary damped wave equation. With this result it is sufficient to produce sufficiently good quasimodes, to do so we reduce the problem to one dimension and then the interval $[0,b]$. We will then show that we can use solutions to a related problem, the complex absorbing potential, on the half line to produce the desired quasimodes. We finally show that such solutions of the complex absorbing potential problem exist by producing solutions on $(0,a)$ and $(a,\infty)$ separately, the latter following from a rescaling argument, we are able to glue these solutions together via a compatibility condition which we satisfy via the implicit function theorem. 

\begin{proposition}\label{BTresult}
Fix $\alpha$, if there exist sequences $\{q_j\} \in \C, \{u_j\} \in H^2(M) \cap H_0^1(M),$ (or $H^2(M)$ if $\p M= \emptyset$) and some $j_0 \in \Nb$, such that for all $j>|j_0|$ 
\begin{equation}\label{quasimode}
\nm{-\Del u_j + i q_j W(x) u_j- q_j^2 u_j}_{L^2(M)}^2 \leq \frac{C}{ |Re (q_j)|^{2/\alpha}}\left( \nm{u_j}_{H^1(M)}^2 + |q_j|^2 \nm{u_j}_{L^2(M)}^2 \right),
\end{equation}
and 
\begin{equation}\label{qposition}
|q_j| \ra \infty, \quad |\operatorname{\operatorname{Im}}(q_j)| \leq \frac{C}{|Re(q_j)|^{1/\alpha}},
\end{equation}
then for all $\e>0$ the system \eqref{DWE} is not stable at rate 
$$
1/t^{\alpha + \e}.
$$
\end{proposition}
\textbf{Remark.} Although Proposition \ref{BTresult} has $\hp{u_j}{1}^2$  on the right hand side of \eqref{quasimode} the quasimodes $u_j$ we will apply it to satisfy a stronger inequality, namely 
\begin{equation}\label{realquasimode}
\nm{-\Del u_j + i q_j W(x) u_j- q_j^2 u_j}_{L^2(M)}^2 \leq \frac{C}{|Re(q_j)|^{2}}\nm{u_j}_{L^2(M)}^2.
\end{equation}
The strength of the conclusion we obtain from applying Proposition \ref{BTresult} to these quasimodes is instead limited by the $q_j$ for which we have such an estimate due to \eqref{qposition}.

Note that producing sequences $q_j$ and $u_j$ which satisfy the hypotheses of this proposition with $\alpha =\frac{\beta+2}{\beta+3}$ proves Theorem \ref{mainresult}.

We will make two simplifications before proceeding.  First we will reduce the problem to obtaining quasimodes of an ordinary differential equation on $[-b,b]$. We will then further restrict our attention to the same equation on $[0,b]$. 
After making these simplifications we will introduce three key parameters and the complex absorbing potential problem on $(0,\infty)$, solutions of which we will use to produce our desired quasimodes. 

For the first simplification note that for any sequence of integers $m_j$ if $\ti{u_j}$ is a sequence of functions on $[-b,b]$ which satisfy  
\begin{equation}\label{1dwe}
\begin{cases}
\nm{-\p_x^2 \ti{u_j}+ i q_j W \ti{u_j}+ \left(\frac{4 \pi^2 m_j^2}{b^2} - q_j^2\right) \ti{u_j}}_{L^2(-b,b)}^2 \leq \frac{C}{|Re(q_j)|^2} \nm{\ti{u_j}}_{L^2(-b,b)}^2 \quad \text{ as } j \ra \infty \\
\ti{u_j}(x)=0 \quad |x|=b,
\end{cases}
\end{equation}
then $u_j(x,y) = \ti{u_j}(x) \sin\left(\frac{2\pi m_j y}{b}\right)$ satisfy \eqref{realquasimode}. Therefore it is enough for us to find functions which satisfy \eqref{1dwe} with $q_j$ which satisfy \eqref{qposition} with $\alpha = \frac{\beta+2}{\beta+3}$.

The second simplification we make is to limit our attention to $[0,b]$ from $[-b,b]$. Since our damping is even, if we find integers $m_j$ and functions $\ti{u_j}$ on $[0,b]$ which satisfy 
\begin{equation}
\label{1dweb}
\begin{cases}
\nm{-\p_x^2 \ti{u_j}+ i q_j W \ti{u_j}+ \left(\frac{4 \pi^2 m_j^2}{b^2} - q_j^2\right) \ti{u_j}}_{L^2(0,b)}^2 \leq \frac{C}{|Re(q_j)|^2} \nm{\ti{u_j}}_{L^2(0,b)}^2 \text{ as } j \ra \infty \\
\ti{u_j}(x)=0 \quad x=b \\
\ti{u_j}(0)=0 \text{ or } \ti{u_j'}(0)=0
\end{cases}
\end{equation}
we can extend the $\ti{u_j}$ to $-b \leq x<0$ by setting $\ti{u}_j(-x)=-\ti{u}_j(x)$ (or $\ti{u}_j(-x)=\ti{u}_j(x)$ resp.) and the resulting functions satisfy \eqref{1dwe}. Therefore it is enough for us to find functions which satisfy \eqref{1dweb} with $q_j$ which satisfy \eqref{qposition} with $\alpha = \frac{\beta+2}{\beta+3}$.

Before we introduce the complex absorbing potential we introduce three new parameters. Let $h \in [0,1)$  be a small parameter which will be sent to 0 and have $\frac{b}{2 \pi h^2} \in \Nb$. Let $l$ be a bounded parameter, when working in the case $u(0)=0$ we impose $l \in \Zb$ and when $u'(0)=0$ we impose $l+\frac{1}{2} \in \Zb$ but otherwise leave $l$ free in relation to $h$. We define 
\begin{equation}\label{lambda}
\lambda_h = \frac{\pi l h}{a} + C_h h^{(\beta+4)/(\beta+2)} \quad C_h = O_h(1) \in \C.
\end{equation}
We will eventually specify $C_h$ more completely in Section \ref{SectionCAPeigen}. When appropriate we will refer to sequences of these parameters as $h_j, \lambda_{h_j}, l_j$ and $C_{h_j}$, where we emphasize that $\lambda_{h_j}$  and $C_{h_j}$ depend on $h_j$. 

Now we introduce the complex absorbing potential problem on $(0,\infty)$ 
\begin{equation}\label{CAP}
\begin{cases}
0 =-h^2 \p_x^2 v + i(x-a)_+^{\beta} v - \lambda_h^2 v \\
v(0)=0 \text{ or } v'(0)=0.
\end{cases}
\end{equation}

In order to relate this to \eqref{1dweb} we make an ansatz for relations between the parameters. If $v_j$ are a sequence of solutions of \eqref{CAP} for some $h_j, l_j, \lambda_{h_j}, C_{h_j}$, we define $q_j$ and $m_j$ as follows 
\begin{align} \label{ansatzq}
m_j&= \frac{b}{2 \pi h_j^2} \in \Nb\\
q_j&= \frac{1}{h_j^2}+ \frac{\lambda_{h_j}^2}{2} = \frac{1}{h_j^2}+ \frac{\pi^2 l_j^2 h_j^2}{a^2} + \frac{2 C_{h_j} \pi l_j }{a} h_j^{(2 \beta +6)/(\beta+2)} + C_{h_j}^2 h_j^{(2 \beta+8)/(\beta+2)}. \nonumber
\end{align}

Note that in this regime 
\begin{align*}
Re(q_j) &= \frac{1}{h_j^2} + O(h_j^2) \quad \operatorname{Im}(q_j) = \frac{2\operatorname{Im}(C_{h_j}) \pi l_j}{a} h_j^{(2 \beta+6)/(\beta+2)} + O(h_j^{(2 \beta+8)/(\beta+2)}),
\end{align*}
so
$$
q_j \ra \infty \text{ and } |\operatorname{Im}(q_j)| \leq \frac{C}{|Re(q_j)|^{(\beta+3)/(\beta+2)}} \quad \text{ as } j \ra \infty.
$$
As we will see shortly, solutions of \eqref{CAP} in this regime satisfy the inequality in \eqref{1dweb} but not necessarily the boundary condition at $x=b$. In order to ensure they do we multiply these solutions by a cutoff function which is $0$ in a neighborhood of $b$. We will see the resulting functions still satisfy the inequality in \eqref{1dweb} as the solutions of \eqref{CAP} in this regime have rapid decay on the support of the potential (see Lemma \ref{dampedltwo}), which is exactly where errors introduced by the cutoff function appear.  

\textbf{Remark} We note also that it is because we are working with solutions to \eqref{CAP} on the half-line that we will only produce quasimodes rather than real modes. It is necessary for us to work on the half-line in order to perform a rescaling that allows us to set $h=0$, if we were working on a finite interval the rescaling would make the interval depend on $h$ which our approach is not well adapted to address. 

Fix $\d>0$ such that $a+\sigma<b-2\d$; we define $\phi \in \Ci(0,\infty)$ to satisfy
\begin{equation}\label{cutoff}
\phi(x) = \begin{cases} 
1 &x< b-2\d \\
0& x>b-\d
\end{cases}
\end{equation}
\begin{proposition}
\label{CAPgoodenough}
Fix $M>0$, let $\{v_j\} \in H^2(0,\infty)$ be a sequence of solutions of \eqref{CAP} with eigenvalues
$$
\lambda_{h_j}= \frac{ \pi l_j h_j}{a} + C_{h_j} h_j^{(\beta+4)/(\beta+2)}, \quad C_{h_j}=O(1) \in \C,
$$
where $ |l_j| \leq M$ and $h_j \ra 0$ as $j \ra \infty$ and $\frac{b}{2\pi h_j^2} \in \Nb$. Set 
$$
u_j(x) = \phi(x) v_j(x).
$$
Then for $j$ large enough so that $h_j < \sigma^{\beta/2}$ the functions $u_j$ with $q_j, m_j$ as defined in \eqref{ansatzq} satisfy \eqref{1dweb} and \eqref{qposition}.
\end{proposition}

It remains to be seen that we can indeed find solutions to the complex absorbing potential problem with eigenvalues of this form.

\begin{theorem}
\label{CAPeigen}
For all $l \in \Zb,$ (or $l +\frac{1}{2} \in \Zb$), there exists $h_0>0$ and $K>0$ such that for all $h \in (0,h_0)$, there exists a $C_h$ with $|C_h|<K$ and $v \in  H^2(0,\infty)\cap H_0^1(0,\infty), v \neq 0$ (resp. $H^2(0,\infty)$) satisfying \eqref{CAP} with $v(0)=0$ (resp. $v'(0)=0$) with $\lambda$ given by \eqref{lambda}.
\end{theorem}

Using Theorem \ref{CAPeigen} we obtain a sequence $\{v_j\}$ of solutions of \eqref{CAP} which satisfy the hypotheses of Proposition \ref{CAPgoodenough} which in turn produces sequences which satisfy the hypotheses of Propositions \ref{BTresult} which in turn proves Theorem \ref{mainresult}.

\textbf{Remark.} It  is straightforward to  extend these results to the case $M=\T^2$. We parametrize $\T^2$ by $[-b,b] \times [-b,b]$ with parallel edges identified. Thus it is enough to show that the quasimodes we produced on the square satisfy periodic boundary conditions and are thus functions on the torus. Our quasimodes are of the form
$$
u(x,y) = v_j(x) \phi(x) \sin\left( \frac{2 \pi m_j y}{b}\right),
$$
so it is straightforward to see they satisfy periodic boundary conditions in $y$ and $x$ (as $u(x,y)\equiv 0$ for $|x-b|<\d$ and $|b+x|<\d$).

We will prove Proposition \ref{CAPgoodenough} in Section \ref{SectionCAPgoodenough}, we will then prove Theorem \ref{CAPeigen} in Section \ref{SectionCAPeigen}. We prove a necessary estimate in Section \ref{proofofdampedltwo}. We finally prove Theorem \ref{lowerbound} in Section \ref{prooflowerbound}.

\section{Proof of Proposition \ref{CAPgoodenough}}
\label{SectionCAPgoodenough}
We begin by stating an estimate necessary for the proof.

\begin{lemma} 
\label{dampedltwo}
Let $v \in H^1(0,\infty)$ be a solution of  \eqref{CAP} with eigenvalue $\lambda= O(h)$ and let $\phi$ be as in \eqref{cutoff}. Fix $s \in \Rb$ then for $h < \sigma^{\beta/2}$ for all $N$ there exists $C_{N,s}>0$ such that 
\begin{align}
\nm{\phi v}_{H^s_h(a+\sigma, b)}^2 &\leq C_{N,s} h^N \ltb{\phi v}^2.
\end{align}
\end{lemma}

This will be proved in Section \ref{proofofdampedltwo} using the semiclassical ellipticity of $-h^2 \p_x^2 + i(x-a)_+^{\beta}-\lambda_h^2$ on $(a+\sigma/4, b)$.

\begin{proof}[Proof of Proposition \ref{CAPgoodenough}]

We have a sequence $v_j$ of solutions of 
$$
0=-h_j^2 \p_x^2 v_j + i (x-a)_+^{\beta} v_j  - \lambda_{h_j}^2 v_j, \quad x \in (0,\infty),
$$
with 
$$
\lambda_{h_j} = \frac{\pi l h_j}{a} + O\left(h_j^{(\beta+4)/(\beta+2)}\right), \quad h_j \ra 0 \text{ as } j \ra \infty.
$$
It is clear that $u_j= \phi v_j$  has $u_j(b)=\phi(b) v_j(b)=0$.  Recalling \eqref{ansatzq} and the subsequent discussion $q_j, m_j$ satisfy \eqref{qposition}. It remains to be seen that $u_j$ satisfies the inequality in \eqref{1dweb}. By \eqref{ansatzq} and \eqref{CAP} $u_j$ satisfies  
\begin{align*}
-\p_x^2 u _j &+ iq_j W(x) u_j + \left(\frac{4 \pi^2 m_j^2}{b^2} - q_j^2\right) u_j \nonumber \\
&= \phi\left( i \frac{\lambda_{h_j}^2}{2} (x-a)_+^{\beta} v_j - \frac{\lambda_{h_j}^4}{4} v_j\right)- \phi'' v_j - 2 \phi' v_j'+ iq_j\left(W(x)-(x-a)_+^{\beta}\right) \phi v_j. 
\end{align*}

Thus 
\begin{align*}
\ltb{-\p_x^2 u _j + iq_j W(x) u_j + \left(\frac{4 \pi^2 m_j^2}{a^2}- q_j^2\right) u_j}^2 \leq &\frac{\lambda_{h_j}^4}{4} \ltb{(x-a)_+^{\beta} u_j}^2 + \frac{\lambda_{h_j}^8}{16} \ltb{u_j}^2 + \ltb{\phi'' v_j}^2 \\
&+ 4 \ltb{\phi' v_j'}^2+ |q_j|^2 \ltb{ (W-(x-a)_+^{\beta}) \phi v_j}^2.
\end{align*}

Since $b-\d>a+\sigma$, by Lemma \ref{dampedltwo} for any $N>0$
\begin{align*}
&\ltb{\phi''v_j}^2 + \ltb{ \phi' v_j'}^2 \leq \frac{C}{h^2} \nm{\phi v_j}_{H_h^2(b-\d,b)} \leq C_N h^N \ltb{\phi v_j}^2. 
\end{align*}

Furthermore by Lemma \ref{dampedltwo} for any $N>0$ 
\begin{align*}
\ltb{(W-(x-a)_+^{\beta}) \phi v_j}^2 &\leq C \nm{ \phi v_j}_{L^2(a+\sigma,b)}^2 dx \leq C_N h^N \ltb{\phi v_j}^2.
\end{align*}

Therefore
\begin{align*}
\ltb{-\p_x^2 u_j + i q_j(x-a)_+^{\beta} u_j+ \left(\frac{4 \pi^2 m_j^2}{b^2}- q_j^2\right) u_j}^2 &\leq C\frac{\lambda_{h_j}^4}{4} \ltb{u_j}^2+ \frac{\lambda_{h_j}^8}{16} \ltb{u_j}^2+C_N h^N \ltb{u_j}^2 \\
&\leq  \frac{C}{|Re(q_j)|^2} \nm{u_j}_{L^2(0,b)}^2.
\end{align*}
Where we used that $|Re(q_j)|= 1/h_j^2 + O(h_j^2)$ and $\lambda_{h_j} = O(h_j)$.
\end{proof}
 
We now show that there are solutions of \eqref{CAP} with the desired eigenvalues.

\section{Proof of Theorem \ref{CAPeigen}}
\label{SectionCAPeigen}

From this point on we focus on solutions of \eqref{CAP} on $(0,\infty)$. We begin by considering the cases when $v(0)=0$ or $v'(0)=0$, but focus on the case $v(0)=0$ for the bulk of the section. We then explain how the proof changes for $v'(0)=0$. 

In order to produce solutions to \eqref{CAP} with the desired eigenvalues we will solve it on $(0,a)$ and $(a,\infty)$ separately. That is given solutions $v_l, v_r \in H^2$ of \eqref{CAP} on $(0,a)$ and $(a,\infty)$ respectively, with the same values of $\lambda_h$ and $h$,  if there exists $B \in \C$ such that 
\begin{equation}
\label{compat}
\begin{cases}
v_l(a) &= B v_r(a) \\
v_l'(a) &= B v_r'(a),
\end{cases}
\end{equation}
then 
$$
v= \begin{cases}
v_l(x) & x<a \\
B v_r(x) & a<x,
\end{cases}
$$
solves \eqref{CAP} on $(0,\infty)$ with the same $\lambda_h$ and $h$ and $v \in H^2(0,\infty) \cap H_0^1(0,\infty)$ (or $H^2(0,\infty)$ if $v_l'(0)=0$) . We will refer to equations \eqref{compat} as the compatibility condition. 

We will explicitly solve \eqref{CAP} on $(0,a)$. We will then use a rescaling of the equation on $(a, \infty)$ and the implicit function theorem to show that the compatibility condition can be satisfied  when $\lambda_h$ is of the form \eqref{lambda} and $h$ is small enough. 

On $(0,a)$ \eqref{CAP} is solved by 
$$
v_l(x) = e^{i \lambda_h(x-a)/h} + \Rl e^{-i\lambda_h (x-a)/h}, 
$$
where we choose $\Rl$ to ensure the boundary condition at $0$ is satisfied. That is 
$$
\Rl =\begin{cases} - e^{-2 i \lambda_h a/h} \quad &v(0)=0, \\
e^{-2i \lambda_h a/h} \quad &v'(0) =0.
\end{cases}
$$
We will work now specify to the case where $v(0)=0$ and work through it in detail and then summarize how the proof changes for $v'(0)=0$.

We now rescale the equation on $(a,\infty)$. If $F$ solves 
\begin{equation}
\label{Fequation}
\begin{cases}
0=-F''(x)+ \left( ix^{\beta} - \frac{\lambda_h^2}{h^{2\beta/(\beta+2)}} \right) F(x)  \qquad x \in (0,\infty)\\
F'(0)=1,
\end{cases}
\end{equation}
then $v_r(x) = F(h^{-2/(\beta+2)} (x-a))$ solves \eqref{CAP} on $(a,\infty)$ with $v_r'(a)=h^{-2/(\beta+2)}$.
This follows immediately from the definition of $F$ and \eqref{CAP}.

\textbf{Remark.} This rescaling is necessary; in order to show the compatibility condition can be satisfied for all $h$ in a neighborhood of 0 we will apply the implicit function theorem and so must be able to set $h=0$. This can not be done in a satisfactory way with solutions of \eqref{CAP}; however for $\lambda_h$ of the form \eqref{lambda}, \eqref{Fequation} is well defined at $h=0$ as 
$$
\frac{\lambda_h^2}{h^{2 \beta/(\beta+2)}} \bigg|_{h=0}= \frac{1}{h^{2\beta/(\beta+2)}} C h^2 + O(h^{(2\beta+6)/ (\beta+2)})\bigg|_{h=0} =0.
$$

Before we apply the implicit function theorem we first establish an inequality for $u \in H^1(0,\infty)$ using some facts about the Neumann spectrum of $-\p_x^2+x^{\beta}$ on $(0,\infty)$. We then use this inequality to establish the existence and uniqueness of $H^2(0,\infty)$ solutions to \eqref{Fequation} and to show the boundary value $F(0)$ is bounded away from $0$ uniformly in the spectral parameter. We then explain how we can use the implicit function theorem to satisfy the compatibility condition and make use of these properties of $F$ to do so.

Let $\sigma_N(-\p_x^2+x^{\beta})$  be the spectrum of $-\p_x^2+x^{\beta}$ on $(0,\infty)$ with Neumann boundary conditions, and let 
$$
\ti{\lambda_1} = \inf \sigma_N(-\p_x^2 + x^{\beta}).
$$ 
\begin{lemma}
$$\ti{\lambda_1} >0$$
and
\begin{equation}\label{varineq}
\ti{\lambda_1} \nm{u}_{L^2(0,\infty)}^2 \leq  \nm{u'}_{L^2(0,\infty)}^2 + \nm{x^{\beta/2} u}_{L^2(0,\infty)}^2,
\end{equation}
for all $u \in H^1(0,\infty)$.
\end{lemma}
\begin{proof}
The Neumann spectrum is discrete (this follows for instance from \cite{HislopSigal2012} Theorems 5.10 and 10.7) so we know that $\ti{\lambda_1}$ is the lowest eigenvalue of the operator. We also know that the spectrum doesn't contain 0, since 
$$
\nm{\p_x u}_{L^2(0,\infty)}^2 +\nm{x^{\beta/2} u}_{L^2(0,\infty)}^2 > 0,
$$ 
for nontrivial $u \in H^1$. Thus $\ti{\lambda_1}>0$. 

The inequality follows immediately from the variational principle for the spectrum of self adjoint operators (see \cite{HislopSigal2012} Corollary 12.2).
\end{proof}

\begin{lemma}\label{Fprop}
For any $|\eta| < \ti{\lambda_1}$, there exists a unique $H^2(0,\infty)$ solution of 
\begin{equation}
\label{Fequation2}
\begin{cases} 
0 = -F''(x) + ix^{\beta} F(x) - \eta F(x) \\
F'(0)=1.
\end{cases}
\end{equation}
Furthermore if we let $F(0,\eta)$ be the value of this function at $x=0$ there exists $C>0$ such that for all $|\eta| \leq \frac{\ti{\lambda_1}}{2}$ 
$$
1/C \leq |F(0,\eta)| \leq C.
$$
and $F(0,\eta)$ is holomorphic in $\eta$ on that same neighborhood.
\end{lemma}
\begin{proof}
We first show the existence of a solution. Let $\psi \in \Cs(0,\infty)$ with $\psi'(0)=1, \psi(0)=0$. Define $Q(\eta, \psi)$ as 
$$
Q(\eta, \psi):= \psi'' - ix^{\beta} \psi + \eta \psi.
$$
Now let $J$ solve
\begin{equation}
\label{Jequation}
\begin{cases}
-J''+ix^{\beta} J - \eta J = Q \\
J'(0) =0,
\end{cases}
\end{equation}
and note that $F= \psi+J$ solves \eqref{Fequation2}. We will apply the Lax-Milgram theorem to show the existence of solutions to \eqref{Jequation}. 

Let 
$$
H = H^1(0,\infty) \cap x^{-\beta/2} L^2(0,\infty),
$$
and define the norm 
$$
\nm{u}_H^2 = \nm{u}_{H^1(0,\infty)}^2 + \nm{x^{\beta/2} u}_{L^2(0,\infty)}^2,
$$
noting that $H$ is a Hilbert space with this norm. 

We define the sesquilinear form $B: H \times H \ra \C$  
$$
B[u,v] = \int_0^{\infty} u' \bar{v}' + ix^{\beta} u \bar{v}  - \eta u \bar{v}.
$$
For any $u,v \in H$
\begin{align*}
|B[u,v]| &\leq \int |u'| |v'| + x^{\beta} |u| |v| + |\eta| |u| |v| \leq C \nm{u}_H \nm{v}_H.
\end{align*}
Furthermore for $u \in H \subset H^1$ using \eqref{varineq} 
\begin{align*}
|B[u,u]| &\geq \int |u'|^2 + x^{\beta} |u|^2 - |\eta| |u|^2 dx \geq \left(1 - \frac{|\eta|}{\ti{\lambda_1}} \right) \int |u'|^2 + x^{\beta} |u|^2 dx  \geq C \nm{u}_H^2. 
\end{align*}
Therefore by Lax-Milgram for any $Q \in H$ there exists a unique $J \in H$ such that 
$$
B[J, v] = \int Q \bar{v} dx,
$$
for all $v \in H$. 

Therefore there exists an $F \in H$ solving \eqref{Fequation2} given by $F= J+\psi$. 

Now to show that $F(0,\eta)$ is holomorphic in $\eta$ we restate the result of our application of Lax-Milgram. We have shown that when $|\eta| < \ti{\lambda_1}$, for all $Q \in H$ there exists a unique $J \in H$ such that
$$
\nm{(-\p_x^2 + ix^{\beta} - \eta)^{-1} Q}_H=\nm{J}_H  \leq C \nm{Q}_H.
$$
Therefore $(-\p_x^2 + ix^{\beta} -\eta)$ is bijective with a bounded inverse for $|\eta| < \ti{\lambda_1}$. Thus the resolvent $(-\p_x^2 + ix^{\beta} -\eta)^{-1}$  exists for $|\eta|< \ti{\lambda_1}$ and by \cite{HislopSigal2012} Theorem 1.2 it is holomorphic in $\eta$ there as well.

Now recall that the trace operator $T: H^1(0,\infty) \ra \Rb$ is linear and continuous on $H$. Thus $T J = J(0,\eta)$ is also holomorphic in $\eta$. Recall that $F(0,\eta)=J(0,\eta)+\psi(0) = J(0,\eta),$ so $F(0,\eta)$ is holomorphic in $\eta$. 

To see that $F$ is unique assume otherwise, so there exists $F_1 \in H^1(0,\infty)$ which also solves \eqref{Fequation2} with $F_1(0) \neq F(0)$. Set $F_2=F_1-F$, then $F_2'(0)=0, F_2 \in H^1(0,\infty)$ and 
$$
0=-\p_x^2 F_2 + ix^{\beta} F_2 - \eta F_2. 
$$
Multiply both sides by $\bar{F_2}$ and integrate then integrate by parts 
$$
0 = \int_0^{\infty} |\p_x F_2|^2 + i x^{\beta} |F_2|^2 - \eta |F_2|^2 dx.
$$
Then using \eqref{varineq} 
$$
0 \geq \int |\p_x F_2|^2 + x^{\beta} |F_2|^2 - |\eta| |F_2|^2 dx \geq \left( 1 - \frac{|\eta|}{\ti{\lambda_1}} \right) \int |\p_x F_2|^2 + x^{\beta} |F_2|^2 dx >0,
$$
since $F_2 \in H^1(0,\infty)$. The inequality is strict since $|\eta| < \ti{\lambda_1}$, which is a contradiction. 

To establish the stated regularity of $F$ note that the equation \eqref{Fequation2} is elliptic and the left hand side is 0 so in fact $F \in H^{\infty}(0,\infty)$ and thus $F \in H^2(0,\infty)$ in particular. 

Now we show that $F(0,\eta)$ is bounded away from $0$ and $\infty$.
The upper bound follows immediately by the holomorphy of $F(0,\eta)$ and the boundedness of $\eta$. 

To see $|F(0,\eta)| > 1/C$ it is enough to show that $F(0,\eta) \neq 0$ since $F$ is continuous and we are considering $\eta$ in a compact set. Assume otherwise, so $F(0,\eta_0)=0$ for some $\eta_0 \in \C, |\eta_0| \leq \ti{\lambda_1}/2$. Multiply both sides of \eqref{Fequation2} by $\bar{F}$ then integrate and integrate by parts
\begin{align*}
0 = \int_0^{\infty} |\p_x F|^2 + ix^{\beta} |F|^2 - \eta_0 |F|^2 dx.
\end{align*}
Then using \eqref{varineq} (noting that here $F \in H_0^1\subset H^1$)
\begin{align*}
0 \geq \int |\p_x F|^2 + x^{\beta} |F|^2 - |\eta_0| |F|^2 dx \geq \left(1 - \frac{|\eta_0|}{\ti{\lambda_1}} \right) \int |\p_x F|^2 + x^{\beta} |F|^2 dx >0,
\end{align*}
We note that the inequality is strict since $|\eta_0| <  \ti{\lambda_1}$, which is a contradiction.
\end{proof}

Now we introduce $\mu_h \in \C$. We use it to clarify the dependence of $\lambda_h$ on $h$ and will implicitly solve for it in terms of $h$ in order to show that the compatibility condition can be satisfied. If $F_0$ is the Dirichlet data of the unique $H^2$ solution of \eqref{Fequation} for $h=\lambda_h=0$, we set $C_1=A_1 + \mu_h$ so our definition of $\lambda_h$ in \eqref{lambda} becomes 
$$
\lambda_h = \frac{\pi l h}{a} + A_1 h^{(\beta+4)/(\beta+2)} + \mu_h h^{(\beta+4)/(\beta+2)}, \quad 
$$
where $l \in \Zb$ and 
$$
A_1 = \frac{\pi l F_0}{a^2}.
$$

Now take $v_r(x)=F(h^{-2/(\beta+2)}(x-a)),$ where $F$ is the unique $H^2$ solution of \eqref{Fequation} with the above $\lambda_h$. Recalling the explicit form of $v_l(x)$, the compatibility condition becomes 
\begin{align*}
\frac{i \lambda_h}{h} (1 - \Rm )  &= B h^{-2/(\beta+2)}\\
1 + \Rm &= B F(0, \mu_h, h),
\end{align*}
where $F(0,\mu_h,h)$ denotes the Dirichlet data of $F$ and $\Rm=\Rl$ and both are written this way to emphasize their dependence on $\mu_h$ and $h$. Divide the top equation by the bottom   
\begin{align*}
\left( \frac{\pi l i}{a} + A_1i  h^{2/(\beta+2)}+i \mu h^{2/(\beta+2)} \right) \left(1 + \exp(-2\pi i l - 2A_1i a h^{2/(\beta+2)} - 2 ia \mu_h h^{2/(\beta+2)}) \right) F(0, \mu_h, h)\\
= h^{-2/(\beta+2)}\left(1 - \exp(-2\pi i l - 2A_1i a h^{2/(\beta+2)} -2 ia \mu_h h^{2/(\beta+2)}) \right).
\end{align*}

Now Taylor expand the exponentials for small $h$  
\begin{align*}
\left( \frac{\pi l i}{a} + A_1i h^{2/(\beta+2)}+i \mu_h h^{2/(\beta+2)} \right) \left(2 - 2A_1 ia h^{2/(\beta+2)} - 2 ia \mu_h h^{2/(\beta+2)} + g(h) \right) F(0, \mu_h, h)\\
= h^{-2/(\beta+2)}\left(2A_1 ia h^{2/(\beta+2)} +2 ia \mu_h h^{2/(\beta+2)} -g(h) \right),
\end{align*}
where $g(h)$ is the remainder term from the Taylor expansion with $g(h)=O(h^{4/(\beta+2)})$. 

So in order to prove Theorem \ref{CAPeigen} it is enough to establish that for all $h$ near $0 \in [0,\infty)$ there exists $\mu_h \in \C$ such that the following function has a zero at $(\mu_h,h)$;
\begin{align}
\label{Gdef}
G(\mu, h) &= \left( \frac{\pi l i}{a} + A_1 i h^{2/(\beta+2)}+i \mu h^{2/(\beta+2)} \right) \left(2 - 2A_1 ia h^{2/(\beta+2)} - 2 ia \mu h^{2/(\beta+2)}) +g(h) \right) F(0,\mu, h) \nonumber \\
&-2A_1 ia -2 ia \mu + g(h) h^{-2/(\beta+2)}.
\end{align}
To do so we apply the implicit function theorem with weak regularity hypotheses to solve for $\mu_h$. We recall the implicit function theorem as stated in Theorem 11.1 of \cite{LoomisSternberg1990} (pp. 166) can be applied if there exists some $h_0$ such that 
\begin{enumerate}
	\item $G(0,0)=0$
	\item $G$ is continuous on $[0,1] \times [0,h_0)$
	\item $D_{\mu} G$ exists and is continuous on $[0,1] \times [0, h_0)$
	\item $D_{\mu}G(0,0)$ is invertible.
\end{enumerate}

To begin we see immediately that 
\begin{align*}
G(0,0)= \frac{2 \pi l i}{a} F_0 - 2 A_1 i a  = \frac{2 \pi l i }{a} F_0 - \frac{2 \pi l i}{a} F_0=0.
\end{align*}
In the following two lemmas we show that points 2 and 3, and 4 are satisfied.  
\begin{lemma}\label{implicitlemma}
There exists $h_0 >0$ such that  $G$ as defined in \eqref{Gdef} is continuous and $\frac{\p}{\p \mu}G$ exists and is continuous on $\{\mu \in \C; |\mu|<1 \} \times [0, h_0)$. 
\end{lemma}

\begin{proof}
To see that $G$ is continuous it will be enough to show that $F(0,\mu,h)$ is continuous as the other terms in $G(\mu, h)$ are clearly continuous in $\mu$ and $h$. Similarly to see that $\frac{\p}{\p \mu} G$ exists and is continuous it is enough to see that $\frac{\p}{\p \mu} F(0,\mu, h)$ exists and is continuous in $\mu$ and $h$. 

We recall that $F(0, \mu_h, h)$ is the Dirichlet data for the $L^2$ solution of \eqref{Fequation2} with spectral parameter
\begin{align*}
\eta &= \frac{\lambda_h^2}{h^{2\beta/(\beta+2)}} = \frac{\pi^2 l^2 h^{4/(\beta+2)} + (2 A_1 \pi l +2\pi l \mu)h^{6/(\beta+2)}}{a} + (A_1^2 + \mu^2 + A_1 \mu)h^{8/(\beta+2)}.
\end{align*}
By Lemma \ref{Fprop} the Dirichlet data for the $L^2$ solution is holomorphic in $|\eta| <\ti{\lambda_1}$. This $\eta$ is a sum of functions which are jointly continuous in $\mu$ and $h$ and continuously differentiable in $\mu$, therefore $F(0,\mu, h)$ is as well. 

Furthermore since $|\mu|<1$ and there is a positive power of $h$ in each term of $\eta$ there exists some $h_0$ such that $|\eta|<\ti{\lambda_1}/2$ for $|\mu|<1$ and $h \in [0, h_0)$. 
\end{proof} 

\begin{lemma}
With $G$ defined as in \eqref{Gdef} 
$$
\frac{\p}{\p \mu} G(\mu, h) \bigg|_{h=0, \mu=0}  \neq 0,
$$
and thus is invertible.
\end{lemma}
\begin{proof}
Note
$$
G(\mu, 0) = \frac{2\pi l i}{a} F(0,\mu,0) -2A_1ai - 2ai \mu
$$
so
$$
\frac{\p}{\p\mu}G(\mu, 0) = \frac{2 \pi l i}{a}\left( \frac{\p}{\p \mu} F(0,\mu,0) \right)-2ai.
$$
When $h=0$ the equation $F$ solves is $0=-F''(x) + ix^{\beta} F(x).$
Therefore when $h=0$ there is no dependence on $\mu$ so $\frac{\p}{\p \mu} F(0,\mu,h=0) = 0$.
Thus 
$$
\frac{\p}{\p \mu} G(\mu=0,h=0) = -2 a i  \neq 0. 
$$
\end{proof}
Now that we have shown we can apply the implicit function theorem we conclude the proof of Theorem \ref{CAPeigen}. That is given some $l$ we let $h_0$ be as in the proof of Lemma \ref{implicitlemma} and let $K=\frac{\pi  |F_0|}{a^2}+1$. Then we choose an $h \in (0,h_0)$ and use the implicit function theorem to produce a $\mu_h \in [0,1]$ such that $G(\mu_h, h)=0$ and $|\mu_h|<1$. We set $C_h=\frac{\pi  F_0}{a^2}+\mu_h$ we are then able to solve \eqref{CAP} on $(0,a)$ and $(a,\infty)$ for $\lambda_h=\frac{\pi l h}{a}+A_1 h^{(\beta+4)/(\beta+2)} + \mu_h h^{(\beta+4)/(\beta+2)}$ such that the compatibility conditions are satisfied giving us a solution $v \in H^2(0,\infty) \cap H_0^1(0,\infty), v \neq 0$ with $|C_h|<K$ and $\lambda$ of the appropriate form. 

\subsection{Case $u'(0)=0$}
We now discuss how these proofs change when $u'(0)=0$. When this is the case  
$$
\Rl = e^{-2 i \lambda_h a/h}.
$$
Because of this we take $l+1/2 \in \Zb$ rather than $l \in \Zb$. This changes the specific steps taken when going from the compatibility condition to the definition of $G$ in \eqref{Gdef} but the eventual definition of $G$ is the same. The specific form of $l$ is otherwise not used so the remaining proofs in this section hold unchanged. 

\section{Proof of Lemma \ref{dampedltwo}}
\label{proofofdampedltwo}
\begin{proof}
To show this result we consider our operator $P_h=-h^2 \p_x^2 + i(x-a)_+^{\beta}-\lambda_h^2$ on $(a+\sigma/4,b)$. We note that the coefficients are smooth away from $a$ and so it makes sense to look at the pricipal symbol
$$
|\xi|^2 + i(x-a)_+^{\beta}.
$$
It is straightforward to see from this that $P_h$ is semiclassically elliptic on $(a+\sigma/4, b)$.

Recall in our definition of $\phi$ that we required $a+\sigma < b-2\d$ and in \eqref{cutoff} that $\phi$ satisfies 
$$
\phi(x) = \begin{cases} 
1 &x< b-2\d \\
0& b-\d<x.
\end{cases}
$$
We now define $\psi \in \Ci(0,b)$ satisfying 
$$
\psi = \begin{cases}
0 & x< a+\sigma/2 \\
1 & a+\sigma<x<b,
\end{cases}
$$
so that $\phi \psi \in \Psi^0_h(0, b)$ with $WF_h(\phi \psi) \subset ell_h(P_h).$

If $v$ is a solution of $P_h v=0$, by standard semiclassical elliptic estimates (see for instance \cite{Zworski2012} Theorem 7.1) there exists $\chi \in \Cs(a+\sigma/4,b)$ such that for all $s \in \Rb$
$$
\nm{\phi \psi v}_{H_h^s(a+\sigma/2, b)} \leq O( h^{\infty}) \nm{\chi v}_{L^2(a+\sigma/4,b)}
$$
In particular
$$
\nm{\phi v}_{H_h^s(a+\sigma, b)} \leq O(h^{\infty}) \nm{v}_{L^2(0,b)}.
$$
It remains to show that 
$$
\nm{v}_{L^2(0,b)} \lesssim \nm{\phi v}_{L^2(0,b)}.
$$
To proceed we multiply both sides of \eqref{CAP} by $\bar{v}$ then integrate and integrate by parts 
$$
0=h^2 \int_0^{\infty} |\p_x v|^2 dx + i \int_0^{\infty} (x-a)_+^{\beta} |v|^2 dx - \lambda_h^2 \int_0^{\infty} |v|^2 dx.
$$
Take the imaginary part of both sides and rearrange
$$
\int_0^{\infty} (x-a)_+^{\beta} |v|^2 dx = \operatorname{Im}(\lambda_h^2) \int_0^{\infty} |v|^2 dx \leq O(h^2) \int_0^{\infty} |v|^2 dx.
$$
Furthermore
$$
\int_{a+\sigma}^{\infty} \sigma^{\beta} |v|^2 dx \leq \int_{a+\sigma}^{\infty} (x-a)_+^{\beta} |v|^2 dx \leq \int_0^{\infty} (x-a)_+^{\beta} |v|^2 dx.
$$
So 
$$
\int_{a+\sigma}^{\infty} |v|^2 dx \leq \frac{h^2}{\sigma^{\beta}} \int_0^{\infty} |v|^2 dx.
$$
Add $\int_0^{\infty} |v|^2 dx$ to both sides and rearrange
$$
\left(1-\frac{h^2}{\sigma^{\beta}}\right) \int_0^{\infty} |v|^2 dx \leq \int_{0}^{a+\sigma} |v|^2 dx.
$$
Notice that $(b-2\d, b) \subset (0, \infty)$ and $(0, a +\sigma) \subset (0,b-2\d)$ so 
$$
\left(1-\frac{h^2}{\sigma^{\beta}}\right) \int_{b-2\d}^b |v|^2 dx \leq \left(1-\frac{h^2}{\sigma^{\beta}}\right) \int_0^{\infty} |v|^2 dx \leq \int_{0}^{a+\sigma} |v|^2 dx \leq \int_0^{b-2\d} |v|^2 dx.
$$
Therefore for $h< \sigma^{\beta/2}$
\begin{align*}
 \int_0^{b} |v|^2 dx = \int_0^{b-2\d} |v| dx + \int_{b-2\d}^b |v|^2 dx \leq \left(1+ \frac{\sigma^{\beta}}{\sigma^{\beta}-h^2}\right) \int_0^{b-2\d} |v|^2 dx \leq C \int_0^b |\phi v|^2 dx.
\end{align*}
\end{proof}

\section{Proof of Theorem \ref{lowerbound}} \label{prooflowerbound}
In this section we establish a rate of decay of the energy by adapting and improving Section 3 of \cite{BurqHitrik2007} for our particular setup. That result and our result rely heavily on an observability result by Burq and Zworski (Proposition 6.1) in \cite{BurqZworski2004}.

To prove Theorem \ref{lowerbound} we again rely on Proposition 2.4 of \cite{AnantharamanLeautaud2014} (see also \cite{BorichevTomilov2010}) which we state a variant of.
\begin{proposition}\label{BTResult2}
If there exists $C>0$ and $q_0 \geq 0$ such that 
\begin{equation}
\nm{ (-\Del + i q W(x) - q^2)^{-1}}_{L^2 \ra L^2} \leq C| q|^{\frac{1}{\alpha}-1}
\end{equation}
for all $q \in \Rb, |q| \geq q_0$ then \eqref{DWE} is stable at rate $1/t^{\alpha}$. 
\end{proposition}
\begin{proof}[Proof of Theorem \ref{lowerbound}]
Consider $u \in H^2(M)$ solving 
\begin{equation}\label{stationaryinhomog}
(-\Del + i q  W - q^2) u = f \in L^2(M), \quad u|_{\p M} = 0, \quad q\gg1.
\end{equation}
Let $0 \leq \chi \in C_0^{\infty}(\Rb)$ be a cutoff function with $\chi=0$ for $|x| \geq 2$ and $\chi =1$ for $|x| \leq 1$. Now let $\chi_q=\chi(q^{\gamma} W(x))$ with $\gamma= \frac{\beta}{\beta+2}$. Note that $\chi_q$ is identically 0 for $|x|>a+\sigma/4$ for $q$ large enough since $W>0$ on $[a+\sigma, b]$. Because of this $\chi_q$ and it's derivatives are only supported where $W$ has the form $(|x|-a)_+^{\beta}$. Furthermore on the support of $\chi'_q(x)$
\begin{equation}
\label{Westimate}
W(x) \sim \frac{1}{q^{\gamma}}, \quad q\gg1.
\end{equation}
\textbf{Remark.} The proof in \cite{BurqHitrik2007} uses an analogous setup with $\gamma=1$. The key change we make is to set $\gamma = \frac{\beta}{\beta+2}$. The rest of our argument is similar to the proof in \cite{BurqHitrik2007} but we detail it for the convenience of the reader and to explain why this value of $\gamma$ is ideal. 

The function $\chi_q u$ still vanishes on $\p M$ (or if $\p M =\emptyset$ it still satisfies the periodicity condition) and satisfies on $M$
\begin{equation}
\label{cutoffu}
(-\Del - q^2) \chi_q u = \chi_q f + \chi_q'' u - 2 \p_x (\chi'_q u) - i q W(x) \chi_q u. 
\end{equation}
We apply Proposition 6.1 of \cite{BurqZworski2004} to this equation choosing the control region $\omega_x = [a+\sigma/4, a+\sigma/2]$ and setting $\omega :=w_x \times [-b,b]$ to obtain 
\begin{align} \label{chinorm}
\lp{\chi_q u}{2}^2 &\leq C \left( \nm{\chi_q f + \chi''_q u - 2 \p_x (\chi'_q u) - i q W(x) \chi_q u}_{H^{-1}_x L^2_y(M)} \right)+ \nm{\chi_q u|_{\omega}}_{L^2(\omega)}^2 \nonumber \\
&\leq C \left( \lp{\chi_q f}{2}^2 + \lp{\chi''_q u}{2}^2 + \lp{\chi'_q u}{2}^2 + \lp{q W \chi_q u}{2}^2 \right).
\end{align}
We emphasize that $\chi_q$ vanishes on $\omega$, which allowed us to drop that term. We now estimate the remaining terms on the right hand side.
Using that $W$ is exactly $(|x|-a)_+^{\beta}$ on the support of $\chi_q$ and its derivatives we obtain the following bound on the derivative of $\chi_q$,
\begin{equation}
\label{chiqcontrol}
|\chi'_q| = |q^{\gamma} \chi'_q(q^{\gamma} W(x)) W'(x)| \leq C q^{\gamma/\beta},
\end{equation}
and similarly
$$
|\chi''_q| \leq C q^{2\gamma/\beta}.
$$
Note that on the support of $\chi'_q$ and $\chi''_q$ the damping $W$ is smooth, so this computation is valid for all $\beta > 0$. 

Now write 
$$
\chi'_q u = \frac{\chi'_q W^{1/2} u}{W^{1/2}},
$$
then using \eqref{chiqcontrol} and \eqref{Westimate} 
$$
\left| \frac{\chi'_q}{W^{1/2}} \right| \leq q^{\gamma(1/2+1/\beta)},
$$
and consequently
\begin{equation}\label{firstchi}
\lp{\chi'_q u}{2}^2 \leq C q^{\gamma(1+2/\beta)} \lp{W^{1/2} u}{2}^2.
\end{equation}
We estimate the $L^2$ norm of $\chi''_q u$ in a similar way 
\begin{equation}\label{secondchi}
\lp{\chi''_q u}{2}^2 \leq O(1) q^{\gamma(1+4/\beta)} \lp{W^{1/2} u}{2}^2.
\end{equation}
Finally a similar argument shows 
\begin{equation}\label{dampingchi}
\lp{q W \chi_q u}{2}^2 \leq O(1) q^{2-\gamma} \lp{W^{1/2} u}{2}^2.
\end{equation}
The smaller the terms on the right of \eqref{firstchi}, \eqref{secondchi} and \eqref{dampingchi} are the stronger the resolvent estimate is. Because of this we would like to minimize 
$$
\max\{2-\gamma, \gamma(1+4/\beta), \gamma(1+2/\beta)\}.
$$
This is attained when 
$$
2-\gamma=\gamma(1+4/\beta),
$$ 
i..e. $\gamma=\beta/(\beta+2)$. 
Therefore \eqref{firstchi}, \eqref{secondchi}, \eqref{dampingchi} along with \eqref{chinorm} give
$$
\lp{\chi_q u}{2}^2 \leq O(1) \left( \lp{f}{2}^2 + q^{(\beta+4)/(\beta+2)} \lp{W^{1/2} u}{2}^2 \right).
$$
Now note that pairing \eqref{stationaryinhomog} with $\bar{u}$ 
\begin{equation}\label{dampingcontrol}
|q| \int W |u|^2 dx \leq \lp{f}{2} \lp{u}{2}.
\end{equation}
Therefore
$$
\lp{\chi_q u}{2}^2 \leq O(1) \left( \lp{f}{2}^2 + q^{2/(\beta+2)} \lp{f}{2} \lp{u}{2} \right).
$$
It remains to control the $L^2$ norm of $(1-\chi_q) u$. To do so we remark that $1-\chi_q$ is supported in the set where $W \geq 1/q^{\gamma}$. Using \eqref{dampingcontrol} again 
$$
\lp{(1-\chi_q) u}{2}^2 \leq q^{\gamma} \int (1-\chi_q) W |u|^2 dx \leq q^{\gamma-1} \lp{f}{2} \lp{u}{2}.
$$
Therefore 
$$
\lp{u}{2}^2 \leq O(1) \left( \lp{f}{2}^2 + q^{1/(\beta+2)} \lp{f}{2} \lp{u}{2} \right)
$$
and thus we obtain
$$
\lp{u}{2} \leq O(1) q^{2/(\beta+2)} \lp{f}{2}, \quad q \in \Rb, \quad |q| \gg1,
$$
which along with Proposition \ref{BTResult2} gives stability at the stated rate. 
\end{proof}

\bibliographystyle{alpha}
\bibliography{mybib}
\end{document}